\documentclass[a4paper,12pt]{article}

\usepackage{amsthm}   \usepackage{amsfonts}     \usepackage{amssymb}
\usepackage{amsmath}  \usepackage{mathrsfs}     \usepackage{mathtools}
\usepackage{url}       \usepackage{fancyhdr}

\usepackage{relsize}
\usepackage[all]{xy}    \usepackage{amscd}   \usepackage{epsfig}

\usepackage{array}   \usepackage{multirow}  \usepackage{booktabs,caption,fixltx2e}
\usepackage[flushleft]{threeparttable}

\usepackage{cases} 
\usepackage{ulem}  
\usepackage{xcolor}  

\usepackage[T1]{fontenc} 
\usepackage[utf8]{inputenc}
\usepackage{authblk}  
\usepackage{lineno}

\newtheorem{theorem}{Theorem}
\newtheorem{lemma}[theorem]{Lemma}

\theoremstyle{definition}
\newtheorem{example}{Example}

\newtheorem{remark}{Remark}
\newtheorem*{case}{Case}
\newtheorem*{step}{Step}

\newcommand{\Aut}{{\mathrm{Aut}}}

\newcommand{\Char}{\hspace{4pt}\mathrm{char}\hspace{4pt}}
\newcommand{\la}{\langle}
\newcommand{\ra}{\rangle}

\begin{document}
\title{Regular dessins uniquely determined by a nilpotent automorphism group}
\author[1,2]{Na-Er Wang\thanks{wangnaer@zjou.edu.cn}}
\author[3,4]{Roman Nedela\thanks{nedela@savbb.sk}}
\author[1,2]{Kan Hu\thanks{hukan@zjou.edu.cn}}
\affil[1]{School of Mathematics, Physics and Information Science, Zhejiang Ocean University, Zhoushan, Zhejiang 316022, People's Republic of China}
\affil[2]{Key Laboratory of Oceanographic Big Data Mining \& Application of Zhejiang Province, Zhoushan, Zhejiang 316022, People's Republic of China}
\affil[3]{University of West Bohemia, NTIS FAV, Pilsen, Czech Republic}
\affil[4]{Mathematical Institute, Slovak Academy of Sciences, Bansk\'a Bystrica, Slovak Republic}
\date{}
\maketitle

\begin{abstract}
It is well known that the automorphism group of a regular dessin is a two-generator finite group, and the isomorphism classes of regular dessins with automorphism groups isomorphic to a given finite group $G$ are in one-to-one correspondence with the orbits of the action of $\Aut(G)$ on the ordered generating pairs of $G$.  If there is only one orbit, then up to isomorphism the regular dessin is uniquely determined by the group $G$ and it is called uniquely regular. In the paper we investigate the classification of uniquely regular dessins with a nilpotent automorphism group. The problem is reduced to the classification of finite maximally automorphic $p$-groups $G$, i.e., the order of the automorphism group of $G$ attains Hall's upper bound. Maximally automorphic $p$-groups of nilpotency class three are classified.
\\[2mm]
\end{abstract}

\section{Introduction}
A \textit{dessin} is a 2-cell embedding of a connected $2$-coloured bipartite graph into a compact
Riemann surface. An automorphism of a dessin is a colour-preserving automorphism of the underlying graph,
regarded as permutations of the edges, which extends to a conformal self-homeomorphism of the supporting
surface. It is well known that the group of automorphisms of a dessin acts semi-regularly on the edges. In the
case where this action is transitive, and hence regular, the dessin will be called regular as well.

Belyi's theorem \cite{Be79} establishes a correspondence between dessins and Riemann surfaces definable
over the field of algebraic numbers, and thus, as observed by Grothendieck~\cite{Gro1997}, develops
a combinatorial approach to the absolute Galois group through its action on dessins. It is recently proved by
Gonz\'alez-Diez and Jaikin-Zapirain that this action remains faithful when restricted to regular dessins~\cite{GJ2013}.
Therefore it is important to investigate regular dessins and the associated quasiplatonic Riemann surfaces,
see~\cite{Jones2013, Jones2012, JS1996,JW2016} and references therein for more details on this subject.

Our ambition is to investigate regular dessins with a maximal group of external symmetries.
In order to explain our motivation we need to introduce a group-theoretical approach to regular dessins.
It is well known that the automorphism group of a regular dessin is a two-generator finite group,
and for a given two-generator finite group $G$, the isomorphism classes of regular dessins with automorphism group isomorphic to $G$ are
in one-to-one correspondence with the orbits of the action of $\Aut(G)$ on the generating pairs of $G$~\cite{Jones2013}.
In the case where $\Aut(G)$ has just one orbit the corresponding dessin will be called
{\it uniquely regular}.

The simplest example of a uniquely regular dessin is determined by the group $\mathrm{C}_n\times \mathrm{C}_n$,
the direct product of two cyclic groups of order $n$. Note that the associated algebraic curve is the well known
Fermat curve defined by the equation $x^n+y^n=z^n$, see \cite{JS1996} for details. It follows from the
uniqueness that uniquely regular dessins are invariant under the action of the group of dessin operations,
and hence possess the highest level of external symmetry. The uniqueness also implies that all uniquely
regular dessins are invariant under the action of the absolute Galois group, and hence the associated
algebraic curves can be defined over the field of rational numbers.

The classification problem of uniquely regular dessins translates to the language of group theory as follows:

\medskip
{\bf Problem 1:} Classify finite two-generator groups $G$ such that $\Aut(G)$ is transitive on the
set of ordered generating pairs of $G$.
\medskip

Though every finite non-abelian simple group $G$ is two-generated~\cite{MSW1994}, it cannot determine a
uniquely regular dessin, since the generating pair of $G$ can be chosen as an involution and a non-involution,
and $G$ does not admit an automorphism transposing them. In this paper we restrict our investigation on Problem 1
to finite nilpotent groups. More precisely, we consider

\medskip
{\bf Problem 2A:} Classify finite two-generator nilpotent groups $G$ such that $\Aut(G)$ is transitive
on the set of ordered generating pairs of $G$.
\medskip

Since every finite nilpotent group is a direct product of its Sylow subgroups, Problem 2A reduces to

\medskip
{\bf Problem 2B:} Classify finite two-generator $p$-groups $G$ such that $\Aut(G)$ is
transitive on the set of ordered generating pairs of $G$.
\medskip

A  solution to Problem 2B (and hence to Problem 2A) is surprisingly nice. Let $G$ be a $d$-generator
$p$-group of order $p^n$, P. Hall showed in~\cite{Hall1933} that
\[
|\Aut(G)|\leq p^{d(n-d)}(p^d-1)(p^d-p)\cdots (p^d-p^{d-1}).
\]
The group $G$ will be called {\it maximally automorphic} if the equality holds. In Theorem~\ref{DESSIN}
we prove that a two-generator $p$-group determines a uniquely regular dessin if and only if it is maximally
automorphic. Although this gives a general answer to Problems~2A and 2B a complete description
of maximally automorphic $p$-groups is not at hand.

Two-generator abelian maximally automorphic $p$-groups are
exactly the homocyclic $p$-groups isomorphic to $\mathrm{C}_{p^n}\times \mathrm{C}_{p^n}$ for some integer $n\geq1$.
Two-generator maximally automorphic $p$-groups of nilpotency class two have been classified in~\cite[Theorem~5]{HNW2015}. The main result of this paper is the following theorem classifying two-generator maximally automorphic $p$-groups of nilpotency class three:
\begin{theorem}\label{CLASSTHREE}
Let $G=\langle x,y\rangle$ be a maximally automorphic $p$-group of nilpotency class three, then $G$ is isomorphic to one of the groups listed below:
\begin{enumerate}
\item[\rm(i)] $p=3$~and~$1\leq c<b= a$ or $1\leq c\leq b\leq a-1$,
\begin{align*}G=\la x,y|&x^{3^a}=y^{3^a}=z^{3^b}=u^{3^c}=v^{3^c}=[x,u]=[x,v]=[y,u]=[y,v]=1, \\
&z=[x,y], u=[z,x], v=[z,y]\ra.
\end{align*}
\item[\rm(ii)] $p>3$~and~$1\leq c\leq b\leq a$,
\begin{align*}G=\la x,y|&x^{p^a}=y^{p^a}=z^{p^b}=u^{p^c}=v^{p^c}=[x,u]=[x,v]=[y,u]=[y,v]=1,\\ &z=[x,y], u=[z,x], v=[z,y]\ra.
\end{align*}
\item[\rm(iii)] $p=2$~and~$1\leq c\leq b\leq a-1,$
\begin{align*}G=\la x,y|&x^{2^a}=y^{2^a}=z^{2^{b}}=u^{2^c}=v^{2^c}=[x,u]=[x,v]=[y,u]=[y,v]=1, \\ &z=[x,y], u=[z,x], v=[z,y]\ra.
\end{align*}
\item[\rm(iv)]$p=2$~and~$1\leq c\leq b\leq a-1,$
\begin{align*}G=\la x,y|&z^{2^{b}}=u^{2^c}=v^{2^c}=[x,u]=[x,v]=[y,u]=[y,v]=1,\\ &x^{2^a}=u^{2^{c-1}},y^{2^a}=v^{2^{c-1}}, z=[x,y], u=[z,x], v=[z,y]\ra.\end{align*}
\item[\rm(v)] $p=2$~and~$1\leq c\leq a-2,$
\begin{align*}
G=\la x,y|&z^{2^{a-1}}=u^{2^c}=v^{2^c}=[x,u]=[x,v]=[y,u]=[y,v]=1,\\
&x^{2^{a-1}}=z^{2^{a-2}},  y^{2^{a-1}}=z^{2^{a-2}},z=[x,y], u=[z,x], v=[z,y]\ra.
 \end{align*}
\item[\rm(vi)] $p=2$~and~$1\leq c\leq a-2,$
\begin{align*}
G=\la x,y|&z^{2^{a-1}}=u^{2^c}=v^{2^c}=[x,u]=[x,v]=[y,u]=[y,v]=1, \\
 &x^{2^{a-1}}=z^{2^{a-2}}u^{2^{c-1}}, y^{2^{a-1}}=z^{2^{a-2}}v^{2^{c-1}},z=[x,y], u=[z,x], v=[z,y]\ra.
 \end{align*}
\end{enumerate}
Moreover, the above groups are pairwise non-isomorphic.
\end{theorem}
In proving the result we employ the theory of group extensions together with some results on metabelian groups.

\section{Maximally automorphic $p$-groups}
Let $G$ be a $d$-generator finite $p$-group of order $p^n$, Hall~\cite{Hall1933} has shown that $|\Aut(G)|$ divides $U(p;n,d)$ where
\[
U(p;n,d)=p^{d(n-d)}(p^d-1)(p^d-p)\cdots (p^d-p^{d-1}).
\]
The group $G$ is called \textit{maximally automorphic} if $|\Aut(G)|=U(p;n,d)$. Berkovich and Janko posed the problem  of studying maximally automorphic $p$-groups in \cite[Research problems and themes I 35(a)]{BJ2008}. Clearly finite  homocyclic $p$-groups $\mathrm{C}_{p^e}^d$ are examples of maximally automorphic $p$-groups.
\begin{example}The quaternion group $\mathrm{Q}_8$ is a non-abelian two-generator $2$-group of
order $2^3$. Since $|\Aut(\mathrm{Q}_8)|=24=U(2;3,2)$, $\mathrm{Q}_8$ is maximally automorphic.
It determines a unique regular dessin, which is the embedding of the doubled $4$-cycle $\mathrm{C}_4^{(2)}$,
the cycle of length $4$ with multiplicity $2$, into the double torus of genus 2. This is depicted in Fig.~\ref{C4(2)},
where the opposite sides of the outer octagon are identified to form the double torus.
\begin{figure}[h!]
  \centering
    \includegraphics[width=0.4\textwidth]{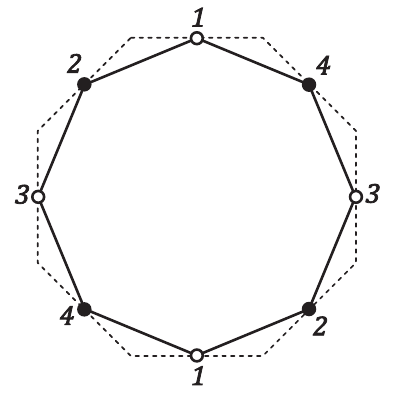}
    \caption{$\mathrm{C}_4^{(2)}$ embedded into double torus}\label{C4(2)}
\end{figure}
\end{example}

\begin{lemma}\label{TRANSITIVE}
A $d$-generator finite $p$-group $G$ of order $p^n$ is maximally automorphic if and only if $\Aut(G)$ is transitive (or equivalently, regular) on the generating $d$-tuples of $G$.
\end{lemma}
\begin{proof}For brevity we denote a $d$-tuple $(x_1,\ldots,x_d)$ by $(x_i)$. Let $T$ be the set of generating $d$-tuples of $G$, that is,
\[
T=\big\{(x_1,x_2,\ldots, x_d)\big|G=\langle x_1,x_2,\ldots, x_d\rangle\big\}.
\]
By Burnside's Basis Theorem $\bar G=G/\Phi(G)$ is elementary abelian of rank $d$ where $\Phi(G)$ is the Frattini subgroup of $G$. Regarding $\bar G$ as a linear space of dimension $d$ over the Galois field $\mathbb{F}_p$, a generating $d$-tuple of $\bar G$ is a base of the linear space. Thus the number of generating $d$-tuples of $\bar G$ is equal to $|GL(d,p)|= (p^d-1)(p^d-p)\cdots (p^d-p^{d-1})$.

Moreover, each generating $d$-tuple $(\bar x_i)$ of $\bar G$ lifts to precisely $|\Phi(G)|^d=p^{d(n-d)}$ generating $d$-tuples $\big\{(x_ig_i)\mid g_i\in\Phi(G)\big\},$ and each generating $d$-tuple of $G$ arises in this way. Hence $|T|=p^{d(n-d)}|GL(d,p)|=U(p;n,d)$. Since the action of $\Aut(G)$ on $T$ is semiregular, $G$ is maximally automorphic if and only if $\Aut(G)$ acts transitively (or equivalently, regularly) on $T$.
\end{proof}
The following result is an immediate consequence of Lemma~\ref{TRANSITIVE}.
\begin{theorem}\label{DESSIN}
 Let $G$ be a finite two-generator $p$-group, then $G$ determines a uniquely regular dessin if and only if it is maximally automorphic.
\end{theorem}
\begin{proof}
$G$ determines a unique regular dessin if and only if $\Aut(G)$ acts transitively on the generating pairs of $G$, or equivalently, $G$ is maximally automorphic by Lemma~\ref{TRANSITIVE}.
\end{proof}

\begin{lemma}\label{QUOTIENT}
Let $G$ be a $d$-generator finite maximally automorphic $p$-group. If $N\Char G$ and $N\leq \Phi(G)$, then the following statements hold true:
\begin{itemize}
\item[\rm(i)]the quotient group $G/N$ is maximally automorphic,
\item[\rm(ii)]for each $\sigma\in\Aut(G)$, the mapping $\wp:\sigma\mapsto \bar\sigma$ is a group epimorphism from $\Aut(G)$ onto $\Aut(G/N)$ where $\bar\sigma(gN)=\sigma(g)N$,
\item[\rm(iii)]$\ker\wp=C_{\Aut(G)}(G/N)$ is a finite $p$-group of order $|N|^d$, where
\[
C_{\Aut(G)}(G/N)=\{\sigma\in\Aut(G)\mid g^{-1}\sigma(g)\in N\quad\text{for all $g\in G$}\},
\]
\item[\rm(iv)]$\Aut(G)$ is a semidirect product of $C_{\Aut(G)}(G/N)$ by a subgroup $Q\cong\Aut(G/N)$,
\item[\rm(v)]$\Aut(G)$ is transitive on the maximal subgroups of $G$.
\end{itemize}
\end{lemma}
\begin{proof}
Assume $|G|=p^n$, then $|\Phi(G)|=p^{n-d}$. Since $N\leq\Phi(G)$ we may assume $|N|=p^m$ where $m\leq n-d$. By hypothesis $N\Char G$, so each automorphism $\sigma\in\Aut(G)$ of $G$ induces an automorphism $\bar \sigma$ of $G/N$ of the form $\bar\sigma:gN\mapsto \sigma(g)N$, and the mapping $\wp: \Aut(G)\to \Aut(G/N), \sigma\mapsto \bar \sigma$ is a group homomorphism.

Moreover, for any generating $d$-tuple $(x_i)$ of $G$, define a set of $d$-tuples as
\[
\Delta=\{(x_1g_1,x_2g_2,\ldots,x_dg_d)\big|~\text{each}~g_i\in N\}.
\]

Since $N\leq \Phi(G)$, each $d$-tuple in $\Delta$ generates $G$, and the group $K:=\ker\wp$ acts semiregularly on $\Delta$. Thus $|K|$ divides $|\Delta|=|N|^d=p^{md}$, and hence
\[
|\Aut(G/N)|\geq |\Aut(G)|/|K|\geq U(p;n,d)/|\Delta|=U(p;n-m,d).
\]
By Hall's theorem $|\Aut(G/N)|$ divides $U(p;n-m,d)$. Thus $|\Aut(G/N)|=U(p;n-m,d)$ and $G/N$ is maximally automorphic.

The above proof also implies that $|\Aut(G/N)|=|\Aut(G)|/|K|$ and $|K|=|\Delta|$, so the mapping $\wp$ is indeed an epimorphism and $K$ is indeed regular on $\Delta$. It follows that the set $T$ consisting of generating $d$-tuples of $G$ splits into $q:=|\Aut(G/N)|$ disjoint blocks of equal size $|K|$. It is easily seen that $K$ fixes every block, and $\Aut(G)$ has a subgroup $Q\cong\Aut(G/N)$ transitive on the blocks, so $K\cap Q=1$ and $\Aut(G)=K\rtimes Q$.

 Finally, since $G/\Phi(G)$ is elementary abelian, $\Aut(G/\Phi(G))$ is transitive on the maximal subgroups of $G/\Phi(G)$. Note that for each maximal subgroup $M$ of $G$ we have $\Phi(G)\leq M$, thus by (ii) $\Aut(G)$ is transitive on the maximal subgroups of $G$, as required.
\end{proof}

\begin{remark}
A finite group $G$ is called an \textit{MI-group} if all maximal subgroups of $G$ are isomorphic. MI-groups were investigated by Hermann \cite{Her81, Her90,Her95} and Mann \cite{Man95}. By Lemma~\ref{QUOTIENT}(v) every maximally automorphic $p$-group is an MI-group.
\end{remark}

\begin{example}Let $p>2$ be a prime, let $G$ be the non-abelian $p$-group of order $p^3$ and of exponent $p$ defined by the presentation
\[
G=\langle x,y\mid x^p=y^p=z^p=[z,x]=[z,y]=1, z:=[x,y]\rangle.
\]
Since $G'\leq Z(G)$ and $\exp(G)=p$, every generating pair $(x',y')$ of $G$ fulfils the above defining relations, so the mapping $x\mapsto x', y\mapsto y'$ extends to an automorphism of $G$. By Lemma~\ref{TRANSITIVE} the group $G$ is maximally automorphic. The associated uniquely regular dessin is of type $(p,p,p)$, embedded into an oriented surface of genus $p^2(p-3)/2+1$ (The reader is referred to \cite{HNW2015} for the formulae of calculating type and genus of a regular dessin). For $p=3$ the dessin is given by a regular embedding of the Pappus graph into the torus, as depicted in Fig.~\ref{Pap}, where the opposite sides of the outer hexagon are identified to form the torus.
\begin{figure}[h!]
  \centering
    \includegraphics[width=0.5\textwidth]{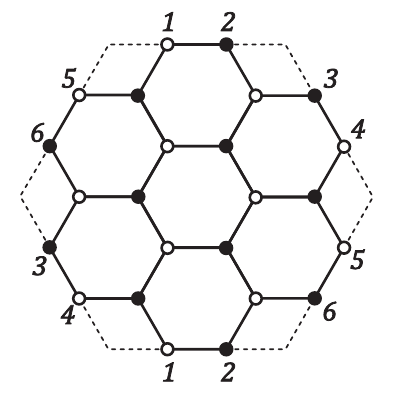}
      \caption{Pappus graph embedded into torus}\label{Pap}
\end{figure}
\end{example}

\begin{example}Let ${F}_d$ denote the free group of rank $d$, $H=\langle g^{p^e}\mid g\in F_d\rangle$, then $H\Char F_d$. Define $B(d,p^e)={F}_d/H$. The group $B(d,p^e)$, called the \textit{Burnside group} of exponent $p^e$ with $d$ generators, is not necessarily a finite group.  Let $K$ be the intersection of all finite-index subgroups of $B(d,p^e)$, then $K\unlhd B(d,p^e)$ and the quotient $R(d,p^e)=B(d,p^e)/K$ is finite by the positive answer to the restricted Burnside problem. By the construction every finite $d$-generator $p$-group of exponent no more than $p^e$ is a homomorphic image of $R(d,p^e)$. In particular, any mapping between two generating $d$-tuples of $R(d,p^e)$ extends to an automorphism of $R(d,p^e)$. Therefore by Lemma~\ref{TRANSITIVE} $R(d,p^e)$ is maximally automorphic. For example the restricted Burnside group $R(2,4)=B(2,4)$ has been investigated by Janko in \cite[\S 60]{BJ2008}. This is a group of order $2^{12}$ and class $5$. It determines a uniquely regular dessin of type $(4,4,4)$, embedded into a surface of genus $2^9+1$.
\end{example}

 Finite two-generator maximally automorphic $p$-groups of class two have been classified in~\cite[Theorem~5]{HNW2015}. The result reads as follows:
\begin{theorem}\label{CLASSTWO}{\rm\cite{HNW2015}}
Let $G$ be a finite two-generator $p$-group of nilpotency class two.  If $G$ is maximally automorphic then $G$ is isomorphic to one of the groups listed below:
\begin{itemize}
\item[\rm(i)]$p$ is odd and $1\leq b\leq a$:
\[
G=\langle x,y\mid x^{p^a}=y^{p^a}=z^{p^b}=[z,x]=[z,y]=1, z=[x,y]\rangle.
\]
\item[\rm(ii)]$p=2$ and $1\leq b\leq a-1$:
\[
G=\langle x,y\mid x^{2^a}=y^{2^a}=z^{2^b}=[z,x]=[z,y]=1, z=[x,y]\rangle.
\]
\item[\rm(iii)]$p=2$ and $a\geq 2$:
\[
G=\langle x,y\mid z^{2^{a-1}}=[z,x]=[z,y]=1, x^{2^{a-1}}=y^{2^{a-1}}=z^{2^{a-2}}, z=[x,y]\rangle.
\]
\end{itemize}
Moreover, the above groups are pairwise non-isomorphic.
\end{theorem}

\section{Classification}
In this section we present a classification of two-generator maximally automorphic $p$-groups of class three. We shall use the standard notation from group theory. In particular, recall that $G^{(1)}=G'=[G,G]$ is the derived subgroup of $G$, and for $i\geq 1$ the $(i+1)$-th derived subgroup of $G$ is defined by induction as  $G^{(i+1)}=[G^{(i)},G^{(i)}]$. Moreover, we denote $G_1=G$, and $G_{i+1}=[G_i,G]$ for $i\geq 1$.

In what follows familiarity with the basic commutator formulae is assumed, see \cite[Chapter III]{Huppert1967}. In particular the following well known properties on metabelian groups will be frequently referred to. For the proof see  \cite{GK08,HK69} or \cite[Proposition 2.1.5]{XQ2010}
\begin{lemma}{\rm\cite{GK08,HK69,XQ2010}}\label{FORM}
Let $G$ be a metabelian group, $x,y,z\in G$. Then the following hold true:
\begin{itemize}
\item[\rm(i)]if $z\in G'$ then $[z,x]^{-1}=[z^{-1},x]$;
\item[\rm(ii)]if $y\in G'$ then $[xy,z]=[x,z][y,z]$ and $[z,xy]=[z,x][z,y]$;
\item[\rm(iii)]for any $x,y,z\in G$, $[x,y^{-1},z]^y=[y,x,z]$;
\item[\rm(iv)]for any $x,y,z\in G$, $[x,y,z][y,z,x][z,x,y]=1$;
\item[\rm(v)]if $z\in G'$ then $[z,x,y]=[z,y,x]$.
\end{itemize}
\end{lemma}
 By induction  it is easy to extend the formula of Lemma~\ref{FORM}(v) as follows. Let
 $x_1,\ldots,x_n$ be arbitrary elements of a metabelian group $G$, then for any $z\in G'$ and for any permutation $\alpha$ of $\{1,2,\ldots, n\}$, we have $[z,x_1,x_2,\ldots,x_n]=[z,x_{\alpha(1)},x_{\alpha(2)},\ldots, x_{\alpha(n)}]$. Therefore for brevity we may denote \[
 [ix,jy]=[x,y, \underbrace{x,\ldots,x}_{i-1},\underbrace{y,\ldots,y}_{j-1}], \]
  where $i$ and $j$ are positive integers.

  To proceed we need more sophisticated formulae on metabelian groups. For the proof see \cite{GK08,HK69} or \cite[Chapter 2]{XQ2010}.
 \begin{lemma}\label{META3}{\rm\cite{GK08,HK69,XQ2010}}
Let $G=\langle x,y\rangle$ be a metabelian group. Then for any integer $s\geq2$,
\[
G_s=\la [ix,(s-i)y], G_{s+1}\mid i=1,2,\cdots,s-1\ra.
\]\end{lemma}

\begin{lemma}\label{META4}{\rm\cite{GK08,HK69,XQ2010}}
Let $G$ be a metabelian group, $x,y\in G$. Then for any positive integers $m$ and $n$
\begin{align}\label{POWER1}
[x^m,y^n]=&\prod_{i=1}^m\prod_{j=1}^n[ix,jy]^{{m\choose i}{n\choose j}}.
\end{align}
\end{lemma}

 \begin{lemma}\label{META5}{\rm\cite{GK08,HK69,XQ2010}}
Let $G$ be a metabelian group, $x,y\in G$. Then for any integer $m\geq2$
\begin{align}\label{POWER2}
(xy^{-1})^m=&x^m\Big(\prod_{i+j\leq m}[ix,jy]^{m\choose{i+j}}\Big)y^{-m}.
\end{align}
\end{lemma}

The following theorem on cyclic extensions of groups is well known.
\begin{theorem}\label{EXTEN}{\rm\cite[Theorem 3.36]{Isaacs2008}}
Let $N$ be a group and $m$ a positive integer, and let $a\in N$ and $\sigma\in\Aut(N)$. If
\[
a^{\sigma}=a\quad\text{and}\quad x^{\sigma^m}=x^a
\]
for all $x\in N$, then there exists a group $G$, unique up to isomorphism, and having $N$ as a normal subgroup
with the following properties:
\begin{itemize}
\item[\rm(i)]$G/N=\langle gN\rangle$ is cyclic of order $m$,
\item[\rm(ii)]$g^m=a$,
\item[\rm(iii)]$x^{\sigma}=x^g$.
\end{itemize}
\end{theorem}

Now we turn to the classification of two-generator maximally automorphic $p$-groups of class three. The following technical result will be useful.
\begin{lemma}\label{AUTO}
Let $G=\langle x,y\rangle$ be a maximally automorphic $p$-group of class three,  denote $z=[x,y]$, $u=[z,x]$ and $v=[z,y]$. Then $G$ is metabelian and each of the assignments $\tau: x\mapsto y, y\mapsto x$, $\pi:x\mapsto x^{-1},  y\mapsto y$ and $\eta: x\mapsto x, y\mapsto yx$ extends to an automorphism of $G$, with the images of $z$, $u$ and $v$ under the corresponding automorphisms summarized in Table~\ref{TAB}
\begin{center}
\begin{threeparttable}[b]
\caption{Three Automorphisms}\label{TAB}
\begin{tabular*}{91mm}[c]{|p{36mm}|p{14mm}|p{14mm}|p{10mm}|}
\toprule
Automorphisms~$\sigma$    & $z^{\sigma}$         &$u^\sigma$            &$v^{\sigma}$\\
\hline
$\tau: x\mapsto y,y\mapsto x$   & $z^{-1}$  & $v^{-1}$   & $u^{-1}$\\
\hline
$\pi: x\mapsto x^{-1},y\mapsto y$   & $z^{-1}u$  & $u$   & $v^{-1}$\\
\hline
$\eta: x\mapsto x,y\mapsto yx$   & $zu$  & $u$   & $uv$\\
\bottomrule
\end{tabular*}
\end{threeparttable}
\end{center}
\end{lemma}
\begin{proof}By hypothesis $G$ is a $p$-group of class three, so $G_4=1$. Since $G^{(2)}\leq G_{4}$~\cite[Theorem 2.12, Chapter III]{Huppert1967}, we have $G^{(2)}=1$, that is, $G$ is metabelian. Note that $G=\langle x,y\rangle=\langle x^{-1},y\rangle=\langle x,yx\rangle$. Since $G$ is maximally automorphic, by Lemma~\ref{TRANSITIVE} each of the above assignments $\tau,\pi$ and $\eta$ extends to an automorphism of $G$.

To calculate the images of $z$, $u$ and $v$ we employ the basic commutator formulae from Lemma~\ref{FORM}. Then
\begin{align*}
z^{\tau}&=[x,y]^{\tau}=[x^{\tau},y^{\tau}]=[y,x]=z^{-1},\\
u^{\tau}&=[z,x]^{\tau}=[z^{\tau},x^{\tau}]=[z^{-1},y]=[z,y]^{-1}=v^{-1},\\
v^{\tau}&=[z,y]^{\tau}=[z^\tau,y^\tau]=[z^{-1},x]=[z,x]^{-1}=u^{-1}.
\end{align*}
Similarly, for $\pi$ we have
\begin{align*}
z^{\pi}=&[x^{-1},y]=[x^{-1},y]^{xx^{-1}}=[y,x]^{x^{-1}}=[y,x][y,x,x^{-1}]\\
=&[x,y]^{-1}[y,x,x]^{-1}=[x,y]^{-1}[x,y,x]=z^{-1}u,\\
u^{\pi}=&[z^{\pi},x^{\pi}]=[z^{-1}u,x^{-1}]=[z^{-1},x^{-1}]=[z,x]=u,\\
v^{\pi}=&[z^{\pi},y^{\pi}]=[z^{-1}u,y]=[z^{-1},y]=v^{-1}.
\end{align*}
Finally, for $\eta$ we have
\begin{align*}
z^{\eta}&=[x,y]^{\eta}=[x,yx]=[x,y]^x=z^x=zu,\\
u^{\eta}&=[zu,x]=[z,x]^u[u,x]=[z,x]=u,\\
v^{\eta}&=[zu,yx]=[z,yx]=[z,x][z,y]^x=uv^x=uv.
\end{align*}
We remark that in the proof we have used the fact that $G_4=1$ and $G_3=\langle u,v\rangle\leq Z(G)$.
\end{proof}

\begin{remark}\label{MARK}
With the notation as in Lemma~\ref{AUTO}, the formulae in Lemma~\ref{META4} and \ref{META5} are reduced to the following form:
\begin{align*}
[x^m,y^n]&=[x,y]^{mn}[x,y,x]^{n{m\choose2}}[x,y,y]^{m{n\choose2}}=z^{mn}u^{n{m\choose2}}v^{m{n\choose2}},\\
(xy^{-1})^m&=x^m[x,y]^{m\choose 2}[x,y,x]^{m\choose3}[x,y,y]^{m\choose3}y^{-m}.
\end{align*}
Replacing $y^{-1}$ by $y$ in the second identity we obtain
\[
(xy)^m=x^m[x,y^{-1}]^{m\choose2}[x,y^{-1},x]^{m\choose3}[x,y^{-1},y^{-1}]^{m\choose3}y^m.
\]
Since $[x,y^{-1}]=z^{-1}v$, $[x,y^{-1},x]=u^{-1}$ and $[x,y^{-1},y^{-1}]=v$ we get
\begin{align}\label{PRT}
(xy)^m=x^mz^{-{m\choose2}}u^{-{m\choose3}}v^{{m\choose2}+{m\choose3}}y^m.
\end{align}
Applying $\tau$ to Eq.~\eqref{PRT} we obtain
\begin{align}\label{PRT2}
(yx)^m=y^mz^{{m\choose2}}u^{-{m\choose2}-{m\choose3}}v^{{m\choose3}}x^m.
\end{align}
\end{remark}

\begin{remark}
The automorphisms $\tau$, $\pi$ and $\eta$ correspond to three types of dessin operations studied in \cite{JP2010}: the first swaps the black and the white vertices, the second is the Petrie duality operation, and the last interchanges the black vertices and faces. It was shown that the three operations generate the entire group $\Omega$ of dessin operations which is isomorphic to $GL(2,\mathbb{Z})$.
\end{remark}

\begin{lemma}\label{CTG}With the same hypothesis and notation as Lemma~\ref{AUTO}, $G_3$ is a homocyclic $p$-group of rank two with a presentation
 \begin{equation*}
G_3=\langle u,v\mid u^{p^c}=v^{p^c}=[u,v]=1\rangle\cong \mathrm{C}_{p^c}^2\quad\text{for some $c\geq1$.}
\end{equation*}
\end{lemma}

\begin{proof}By Lemma~\ref{META3} we have $G_3=\langle u,v\rangle$. Assume that $\langle u\rangle\cap\langle v\rangle=\langle u^i\rangle$, then $u^i=v^j$ for some integer $j$. Applying $\eta$ to the relation we have $u^i=(uv)^j$, so $v^j=u^i=(uv)^j=u^jv^j$, and hence $u^j=1$; since $\tau(u)=v^{-1}$, we have $o(u)=o(v)=p^c$ for some integer $c\geq0$. Consequently $v^j=1$, whence $\langle u\rangle\cap\langle v\rangle=1$. Since $G_3>1$, we have $c\geq1$. Therefore $G_3$ has the claimed presentation.
\end{proof}

\noindent\textbf{Proof of Theorem \ref{CLASSTHREE}:}
By hypothesis $G$ is maximally automorphic of class three. Since $G_3\leq\Phi(G)$ and $G_3\Char G$, by Lemma~\ref{QUOTIENT} the quotient $\bar G=G/G_3$ is maximally automorphic of class two. It follows that $\bar G$ is one of the groups listed in Theorem~\ref{CLASSTWO}. Denote $u=[x,y,x]$ and  $v=[x,y,y]$. Then by Lemma~\ref{CTG} $o(u)=o(v)=p^c$, $c\geq1$, and $G_3=\langle u,v\mid u^{p^c}=v^{p^c}=[u,v]=1\rangle\cong \mathrm{C}_{p^c}^2$. Therefore $G$ is a central extension of a homocyclic $p$-group $G_3\cong\mathrm{C}_{p^c}^2$ by a maximally automorphic $p$-group of class two. In what follows, we prove the result in three steps.

\begin{step}[1]  Determination of the presentation of $G$. \par
We distinguish two cases according to the presentation of $\bar G$ listed in Theorem~\ref{CLASSTWO}.
\begin{case}[A]$\bar G$ has a presentation of the form
\[
\bar G=\langle \bar x,\bar y\mid \bar x^{p^a}=\bar y^{p^a}=\bar z^{p^b}=[\bar z,\bar x]=[\bar z,\bar y]=1,\bar z=[\bar x,\bar y]\rangle,
\]
where $p$ is a prime, $p\geq 2$. We assume that
\begin{align}
& z^{p^b}=u^iv^j,\label{PEQN1}\\
& x^{p^a}=u^rv^s,\label{PEQN2}
\end{align}
where $i,j,r,s\in\mathbb{Z}_{p^c}$. Then $u^{p^b}=[z,x]^{p^b}=[z^{p^b},x]=[u^iv^j,x]=1,$ so $c\leq b$.
Applying the automorphisms $\tau$ and $\eta$ to Eq.~\eqref{PEQN1} we obtain $z^{p^b}=u^jv^i$ and $(zu)^{p^b}=u^{i+j}v^j$. Since $c\leq b$, the latter is reduced to $z^{p^b}=u^{i+j}v^j$. Combining these relations with Eq.~\eqref{PEQN1} yields $u^{i-j}=v^{i-j}$ and $u^j=1$, so $i\equiv j\equiv0\pmod{p^c}$, which implies that $o(z)=p^b$ and $\langle z\rangle\cap G_3=1$. Moreover, by Lemma~\ref{META4} and using substitution for $x^{p^a}$ in Eq.~\eqref{PEQN2} we have
\[
1=[u^rv^s,y]\stackrel{\eqref{PEQN2}}=[x^{p^a},y]\stackrel{\eqref{POWER1}}=z^{p^a}u^{p^a\choose2},
 \]
 so $z^{p^a}=u^{-{p^a\choose2}}$.
 Since $\langle z\rangle\cap G_3=1$,  we have $z^{p^a}=u^{-{p^a\choose2}}=1$, whence $b\leq a$ and
\begin{align}
{p^a\choose 2}\equiv0\pmod{p^c}.\label{EQN1}
\end{align}

Moreover, applying $\pi$ and $\eta$ to Eq.~\eqref{PEQN2} we have $x^{p^a}=u^{-r}v^{s}$ and $x^{p^a}=u^{r+s}v^s.$
Combining these relations with Eq.~\eqref{PEQN2} yields $u^{2r}=1$ and $u^s=1$, so $s\equiv0\pmod{p^c}$ and
\begin{align}
2r\equiv0\pmod{p^c}.\label{EQN5}
\end{align}
Therefore Eq.~\eqref{PEQN2} is reduced to $x^{p^a}=u^r$. Applying $\tau$ to this relation we get $y^{p^a}=v^{-r}$, and applying $\eta$ to the latter equation yields   $(yx)^{p^a}=u^{-r}v^{-r}$. So by Lemma~\ref{META5} we have
\begin{align*}
u^{-r}v^{-r}&=(yx)^{p^a}\stackrel{\eqref{PRT2}}=y^{p^a}z^{p^a\choose2}u^{-{p^a\choose2}-{p^a\choose3}}v^{p^a\choose3}x^{p^a}\\
&=z^{p^a\choose 2}u^{r-{p^a\choose2}-{p^a\choose3}}v^{{p^a\choose3}-r}\stackrel{\eqref{EQN1}}=z^{p^a\choose2}u^{r-{p^a\choose3}}v^{{p^a\choose3}-r}.
\end{align*}
This is reduced to $z^{p^a\choose2}=u^{-2r+{p^a\choose3}}v^{-{p^a\choose3}}\stackrel{\eqref{EQN5}}=u^{{p^a\choose3}}v^{-{p^a\choose3}}.$  Since $\langle z\rangle\cap G_3=\langle u\rangle\cap\langle v\rangle=1$, we get
\begin{align}
{p^a\choose2}\equiv 0\pmod{p^b},\label{EQN7}\\
{p^a\choose3}\equiv 0\pmod{p^c}.\label{EQN6}
\end{align}
If $p>2$, then by Eq.~\eqref{EQN5} we get $r\equiv0\pmod{p^c}$; in particular, if $p=3$, then by Eq.~\eqref{EQN6} we have $c\leq a-1$. Consequently we obtain the groups in (i) and (ii). On the other hand, if $p=2$, then by Eq.~\eqref{EQN7} we get $b\leq a-1$, and by  Eq.~\eqref{EQN5} either $r\equiv0\pmod{2^c}$ or $r\equiv 2^{c-1}\pmod{2^c}$, corresponding to the groups in (iii) and (iv) respectively.
\end{case}

\begin{case}[B]$\bar G$ has a presentation of the form
\[
\bar G=\langle \bar x,\bar y\mid \bar z^{2^{a-1}}=[\bar z,\bar x]=[\bar z,\bar y]=1,  \bar x^{2^{a-1}}=\bar y^{2^{a-1}}=\bar z^{2^{a-2}},\bar z=[\bar x,\bar y]\rangle.
\]
As before we assume that
\begin{align}
&z^{2^{a-1}}=u^iv^j,\label{EQN8}\\
&x^{2^{a-1}}=z^{2^{a-2}}u^rv^s,\label{EQN9}
\end{align}
where $i,j,r,s\in\mathbb{Z}_{2^{c}}$. Since $u^{2^{a-1}}=[z^{2^{a-1}},x]=[u^iv^j,x]=1$,  we have $c\leq a-1$. Applying $\tau$ and $\eta$ to Eq.~\eqref{EQN8} we get $z^{2^{a-1}}=u^jv^i$ and $z^{2^{a-1}}=u^{i+j}v^j$. Combining these with Eq.~\eqref{EQN8} we obtain $u^{i-j}=v^{i-j}$ and $u^j=1$, so $i\equiv j\equiv0\pmod{2^{c}}$, and hence $o(z)=2^{a-1}$ and $\langle z\rangle\cap G_3=1$.

Moreover, by Lemma~\ref{META4} we have $[x^{2^{a-1}},y]=z^{2^{a-1}}u^{2^{a-1}\choose 2}=u^{2^{a-1}\choose 2}$, and by Eq.~\eqref{EQN9} we have $[x^{2^{a-1}},y]=[z^{2^{a-2}}u^rv^s,y]=[z^{2^{a-2}},y]=v^{2^{a-2}}$, so $u^{{2^{a-1}\choose2}}=v^{2^{a-2}}$. Since $\langle u\rangle\cap \langle v\rangle=1$, we get $u^{{2^{a-1}\choose2}}=v^{2^{a-2}}=1$, so $c\leq a-2$. Applying  $\eta$ to Eq.~\eqref{EQN9} yields
$x^{2^{a-1}}=z^{2^{a-2}}u^{r+s}v^s.$ Combining this with Eq.~\eqref{EQN9} we get $u^{s}=1$, so $s\equiv0\pmod{2^c}$. Hence Eq.~\eqref{EQN9} is reduced to $x^{2^{a-1}}=z^{2^{a-2}}u^r$. By applying $\pi_1=\tau\pi\tau:x\mapsto x, y\mapsto y^{-1}$ this relation is transformed to $x^{2^{a-1}}=z^{2^{a-2}}u^{-r}$. Combing these two relations we get $z^{2^{a-2}}u^r=x^{2^{a-1}}=z^{2^{a-2}}u^{-r}$, so $u^{2r}=1$, and hence $2r\equiv0\pmod{2^c}$. It follows that either $r\equiv0\pmod{2^c}$ or $r\equiv2^{c-1}\pmod{2^c}$, corresponding to the groups in (v) and (vi) of Theorem~\ref{CLASSTHREE}.
\end{case}
\end{step}
It remains to show that in each case the group $G$ given by the presentation is the desired extension, provided that the numerical conditions are satisfied. We will demonstrate the proof for Case (i) of Theorem~\ref{CLASSTHREE} and leave other cases to the reader. We start with an abelian group $N$ defined by the presentation
 \[
 N=\langle u,v,z\mid u^{3^c}=v^{3^c}=z^{3^b}=[u,v]=[u,z]=[v,z]=1\rangle,
 \]
 where $1\leq c\leq b$. Add an element $x$ to $N$ by
 \[
 x^{3^a}=1, u^x=u,v^x=v, z^x=zu,
  \]
  where $c\leq a$. Then by Theorem \ref{EXTEN} it is easily verified that $H=\langle N,x\rangle$ is an extension of $N$ by $\langle x\rangle\cong\mathrm{C}_{3^a}$. Moreover, add an element $y$ to $H$ by
  \[
  y^{3^a}=1, u^y=u, v^y=v,z^y=zv, x^y=xz,
  \]
  where $b\leq a$. By Theorem \ref{EXTEN} again $K=\langle H,y\rangle$ is an extension of $H$ by $\langle y\rangle\cong\mathrm{C}_{3^a}$. Thus $K$ is a finite group with the presentation
  \[
  \begin{aligned}
  K=\langle u,v,z,x,y\mid &u^{3^c}=v^{3^c}=z^{3^b}=[u,v]=[u,z]=[v,z]=1,\\
  &x^{3^a}=1, u^x=u,v^x=v, z^x=zu, \\
  &y^{3^a}=1, u^y=u, v^y=v,z^y=zv, x^y=xz\rangle,
  \end{aligned}
  \]
  where the numerical condition is $1\leq c\leq b\leq a$. Observe that in the presentation of $K$
   the relations $[u,v]=[u,z]=[u,z]=1$ can be derived from the others, so can be deleted, and the relations $z^x=zu$, $u^x=u$, $v^x=v$, $u^y=u$, $v^y=v$, $x^y=xz$, $z^x=zu$ and $z^y=zv$ can be rewritten as $[x,u]=1$, $[x,v]=1$, $[y,u]=1$, $[y,v]=1$, $z=[x,y]$, $u=[z,x]$ and $v=[z,y]$, respectively. In  particular if $1\leq c<b=a$ or $1\leq c\leq b\leq a-1$ then we have $K=G$ (the reason why we have more restrictive numerical conditions is that $G$ is maximally automorphic, as one can see from the preceding proof). In particular, the defining relations give true orders of the generating elements. We remark that one may employ an alternative approach to verify the existence of such an extension, see \cite[Chapter 9]{Sims1994}.

\begin{step}[2] Proof that $G$ is maximally automorphic.\par
Assume that $(x_1,y_1)$ is an arbitrary generating pair of $G$. Then $x_1$ and $y_1$ can be written as the form $x_1=x^iy^jz^ku^mv^n$ and $y_1=x^ry^sz^tu^hv^f$, where $i,j,k,m,n,r,s,t,h,f$ are integers. By Burnside's Basis Theorem we have $si-rj\not\equiv0\pmod{p}$. We need to show that the generating pair $(x_1,y_1)$ fulfils all the defining relations, so by von Dyck's Theorem \cite{LS1977} the assignment $x\mapsto x_1, y\mapsto y_1$ extends to an epimorphism from $G$ onto itself, which therefore must be an automorphism of $G$. We will demonstrate this for Case~(i) of Theorem~\ref{CLASSTHREE}, and leave the verification for other cases to the reader.

 Denote $z_1=[x_1,y_1]$, $u_1=[z_1,x_1]$ and $v_1=[z_1,y_1]$. Bearing formulae in Remark~\ref{MARK} in mind we do the following calculation:
\begin{align*}
z_1&=[x^iy^jz^ku^mv^n,x^ry^sz^tu^hv^f]=[x^iy^jz^k,x^ry^sz^t]\\
&=[x^iy^j,x^ry^s][x^iy^j,z^t][z^k,x^ry^s]\\
&=[x^i,x^ry^s]^{y^j}[y^j,x^ry^s][x^iy^j,z]^t[z,x^ry^s]^k\\
&=[x^i,y^s]^{y^j}[y^j,x^r]^{y^s}[x,z]^{it}[y,z]^{jt}[z,x]^{rk}[z,y]^{sk}\\
&=(z^{is}u^{s{i\choose2}}v^{i{s\choose2}})^{y^j}(z^{-jr}u^{-j{r\choose2}}v^{-r{j\choose2}})^{y^s}u^{kr-it}v^{ks-jt}\\
&=z^{is-jr}u^{kr-it+s{i\choose2}-j{r\choose2}}v^{ijs-jrs+ks-jt+i{s\choose2}-r{s\choose2}};\\
u_1&=[z^{is-jr}u^{kr-it+s{i\choose2}-j{r\choose2}}v^{ijs-jrs+ks-jt+i{s\choose2}-r{s\choose2}},x^iy^jz^ku^mv^n]\\
&=[z^{is-jr},x^iy^j]=[z,x^iy^j]^{is-jr}=([z,x^i][z,y^j])^{is-jr}\\
&=u^{i(is-jr)}v^{j(is-jr)};\\
v_1&=[z^{is-jr}u^{kr-it+s{i\choose2}-j{r\choose2}}v^{ijs-jrs+ks-jt+i{s\choose2}-r{s\choose2}},x^ry^sz^tu^hv^f]\\
&=[z^{is-jr},x^ry^s]=[z,x^ry^s]^{is-jr}=[z,x^r]^{is-jr}[z,y^s]^{is-jr}\\
&=u^{r(is-jr)}v^{s(is-jr)}.
\end{align*}
It is clear that $z_1^{3^b}=u_1^{3c}=v_1^{3^c}=[x_1,u_1]=[x_1,v_1]=[y_1,u_1]=[y_1,v_1]=1$. It remains to show $x_1^{3^a}=y_1^{3^a}=1$. Note that $\exp(G_2)=3^b$ and $\exp(G_3)=3^c$ where either $c<b=a$ or $c\leq b\leq a-1$. By formula~\eqref{PRT} we have
\begin{align*}
x_1^{3^a}&=(x^iy^jz^ku^mv^n)^{3^a}=(x^iy^jz^k)^{3^a}=(x^iy^j)^{3^a}[x^iy^j,z^{-k}]^{3^a\choose2}z^{k3^a}\\
&=x^{i3^a}[x^i,y^{-j}]^{3^a\choose2}[x^i,y^{-j},x^i]^{3^a\choose3}
[x^i,y^{-j},y^{-j}]^{3^a\choose3}y^{j3^a}u^{ki{3^a\choose2}}v^{kj{3^a\choose2}}=1.
\end{align*}
Similarly $y_1^{3^a}=1$ (We indeed have $\exp(G)=3^a$). Therefore $(x_1,y_1)$ fulfils all defining relations of the group in (i), as required.
\end{step}

\begin{step}[3] Determination of the isomorphism relation.\par
In Table~\ref{TAB2} we summarize the isomorphism classes of $G_3$, $G'$ and $G^{\mathrm{ab}} =G/G'$ for $G$ from each of the six families.
\begin{center}
\begin{threeparttable}[b]
\caption{Invariant Types of $G_3$, $G'$ and $G^{\mathrm{ab}}$}\label{TAB2}
\begin{tabular*}{104mm}[c]{|p{10mm}|p{10mm}|p{20mm}|p{10mm}|p{32mm}|}
\toprule
Case                   &$G_3$     & $G'$         &$G^{\mathrm{ab}}$ & Condition\\
\hline
(i)&$\mathrm{C}_{3^c}^2$  &$\mathrm{C}_{3^c}^2\times \mathrm{C}_{3^b}$ & $\mathrm{C}_{3^a}^2$& $1\leq c<b=a$ or $1\leq c\leq b\leq a-1$\\
\hline
(ii)&$\mathrm{C}_{p^c}^2$ &$\mathrm{C}_{p^c}^2\times \mathrm{C}_{p^b}$ & $\mathrm{C}_{p^a}^2$&$1\leq c\leq b\leq a$\\
\hline
(iii)&$\mathrm{C}_{2^c}^2$&$\mathrm{C}_{2^c}^2\times \mathrm{C}_{2^b}$ & $\mathrm{C}_{2^a}^2$&$1\leq c\leq b\leq a-1$\\
\hline
 (iv)& $\mathrm{C}_{2^c}^2$ &$\mathrm{C}_{2^c}^2\times \mathrm{C}_{2^b}$ & $\mathrm{C}_{2^{a}}^2$&$1\leq c\leq b\leq a-1$\\
  \hline
(v) &$\mathrm{C}_{2^c}^2$ &$\mathrm{C}_{2^c}^2\times \mathrm{C}_{2^{a-1}}$ & $\mathrm{C}_{2^{a-1}}^2$&$1\leq c\leq a-2$\\
\hline
 (vi)&$\mathrm{C}_{2^c}^2$&$\mathrm{C}_{2^c}^2\times \mathrm{C}_{2^{a-1}}$ & $\mathrm{C}_{2^{a-1}}^2$&$1\leq c\leq a-2$ \\
\bottomrule
\end{tabular*}
\end{threeparttable}
\end{center}
\end{step}
From the table it is easily seen that the groups are pairwise non-isomorphic, except possibly groups $A$ and $B$ from (v) and (vi), respectively, with $A'\cong B'$ and $A^{\mathrm{ab}}\cong B^{\mathrm{ab}}$. Let $A=\langle x,y\rangle$ and $B=\langle x',y'\rangle$, where the generating pairs satisfy the defining relations in (v) and (vi), respectively. If $A\cong B$ then there is a generating pair $(x_1,y_1)$ of $A$ such that the mapping $\alpha:x'\mapsto x_1,y'\mapsto y_1$ is an isomorphism from $B$ to $A$. Since $A$ is maximally automorphic, the mapping $\beta: x_1\mapsto x, y_1\mapsto y$ is an automorphism of $A$, so the composition $\gamma=\alpha\beta:x'\mapsto x, y'\mapsto y$ is an isomorphism from $B$ to $A$. We have $z^{2^{a-2}}=x^{2^{a-1}}=\gamma({x'}^{2^{a-1}})=\gamma({z'}^{2^{a-2}}{u'}^{2^{c-1}})=z^{2^{a-2}}u^{2^{c-1}},$ so $u^{2^{c-1}}=1$. This is a contradiction since $o(u)=2^c$. Therefore $A\not\cong B$.\qed
\medskip

By Corollary~\ref{DESSIN} each of the maximally automorphic $p$-groups of class three given by Theorem~\ref{CLASSTHREE} determines a unique regular dessin. Their types and genera are summarized in Table~\ref{TAB3}.
\begin{center}
\begin{threeparttable}[b]
\caption{Uniquely Regular Dessins of Class Three}\label{TAB3}
\begin{tabular*}{122mm}[c]{|p{10mm}|p{20mm}|p{30mm}|p{45mm}|}
\toprule
{\rm Case}                     &$|G|$   & {\rm Type}       & {\rm Genus}\\
\hline
{\rm(i)} &$3^{2(a+c)+b}$ & $(3^a,3^a,3^a) $     &$3^{a+b+2c+1}(3^{a-1}-1)/2+1$\\
\hline
{\rm(ii)} &$p^{2(a+c)+b}$ & $(p^a,p^a,p^a) $     &$p^{a+b+2c}(p^a-3)/2+1$\\
\hline
{\rm(iii)}&$2^{2(a+c)+b}$& $(2^a,2^a,2^a) $     &$2^{a+b+2c-1}(2^a-3)+1$\\
\hline
{\rm(iv)}&$2^{2(a+c)+b}$& $(2^{a+1},2^{a+1},2^{a+1}) $     &$2^{a+b+2c-2}(2^{a+1}-3)+1$\\
 \hline
{\rm(v)} &$2^{3a+2c-3}$& $(2^a,2^a,2^a)$      &$2^{2a+2c-4}(2^a-3)+1$\\
 \hline
{\rm(vi)}&$2^{3a+2c-3} $& $(2^a,2^a,2^a) $     &$2^{2a+2c-4}(2^a-3)+1$\\
\bottomrule
\end{tabular*}
\end{threeparttable}
\end{center}

\section*{Acknowledgement}
The authors would like to thank the anonymous referees for helpful comments and suggestions which have improved the content and the presentation of the paper. The first author is supported by Natural Science Foundation of Zhejiang Province (No.~LQ17A010003). The second author is supported by the grants APVV-15-0220, VEGA 1/0150/14, Project LO1506 of the Czech Ministry of Education, Youth and Sports and Project P202/12/G061 of Czech Science Foundation. The third author is supported by Natural Science Foundation of Zhejiang Province (No.~LY16A010010).


\end{document}